\documentclass[english]{article}
\usepackage{amsmath,amssymb,amsthm}
\usepackage{graphics}

\newtheorem{thm}{Theorem}[section]
\newtheorem{cor}[thm]{Corollary}
\newtheorem{lem}[thm]{Lemma}

\theoremstyle{definition}
\newtheorem{defn}[thm]{Definition}
\theoremstyle{remark}

\numberwithin{equation}{section}
\usepackage{enumerate}
\usepackage{graphicx}
\newcommand{\abs}[1]{\left| #1 \right|}
\newcommand{\Edg}[1]{\mathcal{E}\left( #1 \right)}

\begin{document}

\title{Lexicographic Configurations}
\author{Christoph Hering\thanks{
Institute of Mathematics of the University of T\"ubingen, 
Auf der Morgenstelle 10, 72076 T\"ubingen, Germany,\newline email: hering@uni-tuebingen.de} \and Andreas Krebs\thanks{
Wilhelm-Schickard-Institute,  University of T\"ubingen, 
  Sand 13, 72076 T\"ubingen, Germany,\newline
  email: krebs@informatik.uni-tuebingen.de}\and Thomas Edgar\thanks{United States Army, Army Europe Libraries, Stuttgart Library\newline email: thomas.b.edgar@us.army.mil}}

\maketitle

\begin{abstract} We describe a new way to construct finite geometric objects. For every $k$ we obtain a symmetric configuration $\mathcal{E}(k-1)$ with $k$ points on a line. In particular, we have a constructive existence proof for such configurations. The method is very simple and purely geometric. It also produces interesting periodic matrices. 
\end{abstract}

$\bf{Keywords:}$ Configurations, projective planes.

$\bf{MSC ~2010 ~classification:}$ 51E15, 05B15\\

\section{Introduction}\label{Introduction}
In this paper we describe a new geometric way to construct finite projective planes and finite symmetric configurations. It concerns a first choice construction that is very elementary. Notably, it produces an interesting finite incidence geometry $\mathcal{E}(n)$ for every rational integer $n$. $\mathcal{E}(n)$ is a symmetric configuration of order $n$, that is, an incidence geometry $\mathcal{E}(n)=(\mathcal{P},\mathcal{B})$ consisting of a non-empty set $\mathcal{P}$ of elements, which we call \emph{points}, and a set $\mathcal{B}$ of subsets of $\mathcal{P}$, which we call \emph{blocks}, such that
\begin{enumerate}[(i)]
  \item if $b$ and $b^\prime$ are blocks, then $\abs{b\cap b^\prime}\leq 1$,
  \item $\abs{b}=n+1$ for all $b\in\mathcal{B}$, and
  \item every point $p\in\mathcal{P}$ is contained in exactly $n+1$ blocks in $\mathcal{B}$.
\end{enumerate}
Here $\vert \mathcal{P} \vert \geq n^2+n+1$. If $\vert \mathcal{P} \vert = n^2+n+1$ and $n \geq 2$, then $(\mathcal{P},\mathcal{B})$ actually is a projective plane.

In infinitely many cases, our construction leads to the symmetric configuration $\Edg{n}$ very rapidly. We obtain, for example, a completely geometric construction of the projective planes PG(2,16) and PG(2,256) (and thereby of GF(16) and GF(256)) without requiring any algebraic foundations. In the general case, however, the calculations appear to be very long, as we shall see in Section \ref{examples}.

We can, however, easily prove that for every integer $k\geq 1$ there exists a symmetric configuration with $k$ points on each line (allowing a sufficiently large number of points).

In light of the fact that there do not exist projective planes of the orders 6 or 10 (see Section \ref{examples}), we are lead to the problem of determining $\Edg{6}$ and $\Edg{10}$. It might be interesting to know what these symmetric configurations are. Although it is possible and indeed simple to calculate $\Edg{5}$, already for $\Edg{6}$ we needed 3 months of computer time on a usual desktop PC, and, until now, we were not able to determine $\Edg{10}$ (see Section \ref{examples}).

With some further effort, the method can be generalized to general (possibly not symmetric) configurations. Then we obtain a much wider variety of geometries. In particular, we find many more cases which can be finished after a rather short computation. For example, we find point-line geometries of higher dimensional projective spaces, Steiner triple systems and the like (see \cite{KHE}).

Additionally, the construction provides an example of an extremely frugal first choice construction which succeeds efficiently in rather complex situations.

The starting point and many of the results of this paper originate from the Diploma Thesis of one of the authors, Thomas Edgar. We are very grateful to Hans-Joerg Schaeffer for many useful remarks and computational support.

\section{The notation}\label{notation}
$\mathbb{N}$ is the set of positive rational integers. Let $A=(a_{ij})_{i,j\in\mathbb{N}}$ be a matrix over $\{0,1\}$. For $i\in\mathbb{N}$ we denote by $A_{i\ast}$ the $i$-th row of $A$ and by $A_{\ast i}$ the $i$-th column of $A$, i.e.

\[
A_{i\ast}=\left(a_{i1},a_{i2},a_{i3},\ldots\right) \text{ and }
A_{\ast i}=\begin{bmatrix}
  a_{1i} \\
  a_{2i} \\
  a_{3i} \\
  \vdots \\
\end{bmatrix}.
\]

We also define the \emph{weight} of a row and a column:\\

$w\left(A_{i\ast}\right)=\abs{\left\{j\in\mathbb{N}\vert a_{ij}=1\right\}}$ is the \emph{weight} of the row $A_{i \ast }$, and\\

$w\left(A_{\ast i}\right)=\abs{\left\{j\in\mathbb{N}\vert a_{ji}=1\right\}}$ is the \emph{weight} of the column $A_{ \ast i}$.\\

Let $i\in \mathbb{N}$, and assume $w(A_{i \ast})\neq 0$ but finite. Then let 
\[
k = \min\left\{j\in\mathbb{N}\mid a_{ij}=1\right\}\mbox{ and }l = \max\left\{j\in\mathbb{N}\mid a_{ij}=1\right\}
\]
Now $l-k+1$ is the \emph{length} of the row $A_{i\ast}$. Correspondingly we can define the \emph{length} of the column $A_{\ast i}$.

Any pair $(i,j)\in\mathbb{N}\times\mathbb{N}$ is called a \emph{cell} of the matrix $A$. A cell $(i,j)$ is called a
\begin{itemize}
  \item \emph{flag} if $a_{ij} = 1$ and a
  \item \emph{galf} if there exists a flag $(k,l)$ such that $a_{kl}=a_{kj}=a_{il}=1$ and $1\leq k \neq i$ and $1\leq l\neq j$.
\end{itemize}

In the geometric part of the paper, we usually follow the notation in the book of Dembowski \cite{Dem} or the paper \cite{Her}.The \emph{union} of two 

An \emph{incidence structure} is a triple $(\mathcal{P},\mathcal{L},I)$ consisting of two sets $\mathcal{P}$ and $\mathcal{L}$ and a relation $I\subseteq\mathcal{P}\times\mathcal{L}$.
For an incidence structure $(\mathcal{P},\mathcal{L},I)$ we denote
\[
[P]=\left\{l\in\mathcal{L} \mid P\ I\ l\right\} \mbox{ for }P\in\mathcal{P}\mbox{, and}
\]
\[
(l)=\left\{P\in\mathcal{P}\mid P\ I\ l\right\} \mbox{ for }l\in\mathcal{L}\mbox{.}
\]
An incidence structure $(\mathcal{P},\mathcal{L},I)$ is called a \emph{symmetric tactical configuration} if there exists an integer $k$ such that $\abs{(l)}=\abs{[P]}=k$ for all $l\in\mathcal{L}$ and $P\in\mathcal{P}$ (compare Dembowski [1, pp. 4, 5]). A symmetric tactical configuration is called a \emph{symmetric configuration} (with parameters $v_k$) if in addition
\[
\abs{(l)\cap (l^\prime)}\leq 1
\]
for $l, l^\prime\in\mathcal{L}\mbox{ and }l\neq l^\prime$. 
The parameters $v$ and $k$ are defined by $\abs{\mathcal{P}}=v$ and $\abs{(l)}=k$ for all $l\in\mathcal{L}$. (Compare Gropp \cite{Gro}).

Note that the term symmetric configuration is stronger than the term symmetric tactical configuration.\\
\section{The first choice construction}
\subsection{The matrix $A$}
Let $n\in\mathbb{N}$. We construct a matrix $A(n)$ inductively.
\[
A(n)=\left(a_{ij}\mid i,j\in\mathbb{N}\right)
\]
is a $\{0,1\}$-matrix and has the following properties:

\begin{enumerate}[(I)]
 \item There does not exist any pair $((i,j),(k,l))$ of pairs of integers such that
 \[
 a_{ij}=a_{il}=a_{kl}=a_{kj}=1\mbox{, where }i,j,k,l\geq 1,\ i\neq k \mbox{ and }j\neq l.
 \]
 In other words, $A$ does not contain any rectangles, of which all corners are ones.
 \item Every row of $A$ contains at most $n+1$ ones.
 \item Every column of $A$ contains at most $n+1$ ones.
\end{enumerate}
Note that the Axiom (I) is equivalent to 

\begin{enumerate}[(I*)]
 \item There do not exist any $4$ integers $i,j,k,l \in\mathbb{N}$ such that $i\neq j,~k\neq l$ and
 \[
 a_{ik}=a_{il}=a_{jk}=a_{jl}=1.
 \]
\end{enumerate}
Or
\begin{enumerate}[(I**)]
 \item A galf is not a flag of $A$.
\end{enumerate}

We now introduce an inductive construction of $A=(a_{ij})=A(n)$, the \emph{greedy algorithm}. We start with the $0$-matrix.
Assume that $k\geq 1$ and that all rows $A_{i\ast}$ are constructed already for $i< k$.
We construct the row $A_{k\ast}$. Assume that $l\geq 1$ and that $a_{k1},a_{k2},\ldots,a_{k,l-1}$ are constructed already. We denote the matrix that has been constructed by this point as $A^{kl}$ for use later.

Construction of $a_{kl}$: \begin{itemize}
  \item If $\sum_{j<l} a_{kj} \geq n+1$ or $\sum_{i<k} a_{il} \geq n+1$, then $a_{kl}=0$ remains.
  \item Otherwise, we check if there exists a pair $(i,j)$ such that
  \[
  a_{ij}=a_{il}=a_{kj}=1 \mbox{, where } 1 \leq i < k \mbox{ and } 1 \leq j <l.
  \]
  In this case, again $a_{kl}=0$ remains.
  \item If both these conditions are not fulfilled, then we set $a_{kl}=1$.
\end{itemize}
The $k$th row is finished when it contains $n+1$ ones. We shall see below that this is the case after finitely many steps.
\begin{defn}\label{def:3.1}
A row or a column of $A^{kl}$ is called \emph{complete} if its weight is $n+1$.
\end{defn}
\begin{lem}\label{lem:3.1}
Let $k,r\geq 1$ and $\overline A = A^{kr}$. The number of galfs for $\overline A$ of the form $(k,l)$, $l\geq 1$, is at most $xn^2$, where $x$ is the number of ones in $\overline A_{k\ast}$.
\end{lem}
\begin{proof}
Let $(k,l)$ be a galf for $\overline{A}$. By definition there exists a flag $(i,j)$ such that $i<k, ~j\neq l$ and $a_{ij}=a_{il}=a_{kj}=1$. Here $(k,j)$ is one of the flags on the row $A_{k\ast}$. Starting from $(k,j)$ we have at most $n^2$ possibilities for $l$, because there are at most $n$ ones in the column $A_{\ast j}$ apart from $a_{kj}$, and at most $n+1$ ones in each row of $\overline A$. Hence, altogether, there are at most $xn^2$ possibilities for $l$.
\end{proof}
\begin{thm}\label{thm:3.1}
The length of a row $A_{k\ast},~k\geq 1$, of $A$ is less than $2n^3-n(n-3)$.
\end{thm}
\begin{proof}
We consider the construction of a row of $A$. Let $k\geq 1$ and assume that all rows $A_{i\ast}$ are constructed already for $i< k$. We denote the matrix constructed so far by $\widetilde{A}~=~A^{k1}$ and construct the $k$th row $A_{k\ast}$ according to the above construction of $A$. For the first one in this row we must put $a_{kl} =1$, where $l$ is the smallest number such that the $l$th column of $\widetilde{A}$ contains less than $n+1$ ones. A cell between the first one and the last one on $A_{k\ast}$ must be a flag, a galf, or an intersection of $A_{k\ast}$ with a complete column of $\widetilde{A}$. The number of flags on $A_{k\ast}$ will be $(n+1)$.

We estimate the number of galfs: 
By Lemma \ref{lem:3.1}, the row $A_{k*}$ contains at most $(n+1){n^2}$ galfs $(k,i)$. Here we can do a little better: In the stage before we construct the last one in ${A}_{k*}$, say $a_{kr} =1$, we have only $n$ ones in $\overline{A}_{k*}$ ( $\overline{A}~=~A^{kr}$ as above ) and hence at most ${n^3}$ galfs $(k,i)$ by Lemma 3.1 . Therefore there are at most ${n^3}$ galfs of $\overline{A}$ between the first and the last one in the row $A_{k*}$.

Let $\mathcal{C}$ be the set of complete columns $\widetilde{A}_{*j}$ of $\widetilde{A}$ such that $j>l.$ We must find an upper bound for $\vert\mathcal{C}\vert$. To do this, we consider the set $\mathcal{L}$ of rows $A_{i*}$ such that $i<k$ and $a_{ij}=1,$ for some column $A_{*j} \in \mathcal{C}.$ Counting incidences, we find 
\[
\vert \mathcal{C} | \cdot ~(n+1)~\leq ~\vert \mathcal{L} \vert \cdot ~n
\]
(If $A_{i*}\in \mathcal{L}$, then by the construction of the matrix $A,~a_{il}=1$ or $a_{ij}=1$ for some $j<l.$ Hence the row $A_{l*}$ contains at most $n$ ones to the right of the column $A_{*l}$.)

If $A_{i*}\in \mathcal{L},~1\leq i<k,$ then $a_{il}=1,$ or $(i,l)$ is a galf. There are at most ${n^3}$ suitable galfs for the column $\overline{A}_{*l}$, by the dual of Lemma 3.1. 
Hence $\mid \mathcal{L} \mid \leq {n^3}+n$ and $\mid \mathcal{C}\mid \leq ({n^3}+n)n/(n+1)$. This implies that the length of $A_{k*}$ is at most $\mid \mathcal{C} \mid +~{n^3}+n+1={n^3}-{n^2}+2n-2+2/(n+1)+{n^3}+n+1={2n^3}-{n^2}+3n-1+2/(n+1)<{2n^3}-n(n-3)$ for $n\geq 2$. Clearly, the theorem is also true for $n=1$ (see Section \ref{examples}).
\end{proof}
Note that the arguments in the proof of Theorem \ref{thm:3.1} also imply that each row of $A(n)$ contains exactly $n+1$ ones . Also, in the construction of the matrix $A(n)$ described above, the row $A_{k*}$ can be determined after finitely many steps.

\subsection{The right edge is monotonously increasing}
For each $i\geq1$ define $g(i)$ to be the smallest $j$ such that $a_{ij}=1$.
\begin{lem}\label{lem:3.2}
  The function $g$ is monotonously increasing
\end{lem}
\begin{proof} Let $i\geq 1$ and remember the construction of the row $A_{i\ast}$.
Clearly, all the columns $A_{\ast j}$ must be complete for $1 \leq j < g(i)$. Therefore $g(i+1)\geq g(i)$.
\end{proof}
For each $j\geq 1$, define $f(j)$ to be the smallest $i$ such that $a_{ij}=1$.
\begin{lem}\label{lem:3.3}
The function $f$ is monotonously increasing.
\end{lem}
\begin{proof} Suppose that there are $j$ and $k$ such that $1\leq j<k$ and $f(k)<f(j)$. Remember the construction of the row $A_{f(k)\ast}$. When $a_{f(k),j}$ is constructed, we have $\sum_{l<j} a_{f(k),l} < n+1$ as $a_{f(k),k}=1$, and $\sum_{l<f(k)} a_{l,j} = 0$ because, in the column $A_{\ast j}$, we have only zeros above $a_{f(k),j}$.
 Therefore the first and the second condition in our construction in Section 3.1 are not fulfilled. Hence, we must put $a_{f(k),j}=1$ and $f(j)=f(k)$, a contradiction.
\end{proof}
\noindent\textbf{Remark.} Lemma \ref{lem:3.3} also follows from Lemma \ref{lem:3.2} because of the symmetry of $A$. See Theorem \ref{thm:3.2} below.

\subsection{A second, symmetric construction of the matrix $A$}

We introduce an inductive construction of a matrix $C=(c_{ij})$. To start, we set $C=(0)$.

Assume that $k\geq 1$ and that $c_{ij}$ are already constructed for $i,j < k$. We construct the row segment
\[
(c_{k,1}, \ldots, c_{k, k})
\]
and the column segment
\[
\left[
  \begin{array}{c}
    c_{1,k} \\
    \vdots \\
    c_{k,k} \\
  \end{array}
\right]
\]
(again, inductively). Assume that $1\leq l \leq k$, and that $c_{k1},\ldots,c_{k, l-1}$ and $c_{1k},\ldots,c_{l-1, k}$ are constructed already. Denote the matrix constructed so far by $C^{kl} = \overline{C}$ and assume that $\overline{C}$ is symmetric and has the properties (I) - (III).

We call the cell $(k,l)$ admissible if:
\begin{itemize}
  \item There does not exists any pair $(i,j)$ such that $c_{kj}=c_{il}=c_{ij}=1$, $1\leq i < k$ and $1\leq j < l$,
  \item The number of ones in the row $\overline{C}_{k\ast}$ is at most $n$, and
  \item The number of ones in the column $\overline{C}_{\ast l}$ is at most $n$.
\end{itemize}
Also, the cell $(l,k)$ is called admissible if:
\begin{itemize}
  \item There does not exists any pair $(i,j)$ such that $c_{ij}=c_{lj}=c_{ik}=1$, $1\leq i < l$ and $1\leq j < k$,
    \item The number of ones in the row $\overline{C}_{l\ast}$ is at most $n$, and
    \item The number of ones in the column $\overline{C}_{\ast k}$ is at most $n$.
\end{itemize}
Now, because of the symmetry of $\overline{C}$, the cell $(k,l)$ is admissible if and only if $(l,k)$ is admissible. If this is the case, then we put $c_{kl}=c_{lk}=1$, otherwise $c_{kl}=c_{lk}=0$.
Thus we obtain an extended matrix $C^{k,l+1}$. Clearly, $C^{k,l+1}$ again is symmetric and has the properties (II) and (III). Suppose that $C^{k,l+1}$ contains a forbidden rectangle. Then $(k,l)$ or $(l,k)$, w.l.o.g. $(k,l)$, must be a corner of this rectangle. But this is impossible if $(k,l)$ is admissible.
Therefore, the resulting matrix $C^{k,l+1}$ is symmetric and has the properties (I) - (III).

The matrix $C$ actually equals the matrix $A$ which we constructed above. To see this, remember the construction of the coefficients $c_{kl}$ resp. $a_{kl}$ for $1\leq k,l$. In the row segment $(c_{k,1},\ldots,c_{k,k-1})$ of $C_{k\ast}$, the construction of the coefficients $c_{ki}$ equals the construction of the $a_{ki}$ in the construction of the row $A_{k\ast}$ anyway (see Section 3.1). Also, $c_{kk}=a_{kk}$.

Consider the column segment
\[
\left[
  \begin{array}{c}
    c_{1,k} \\
    \vdots \\
    c_{k-1,k} \\
  \end{array}
\right]
\]
and the construction of $c_{lk},~1\leq l <k$. Here, $a_{lk}$ arises in the construction of the row $A_{l\ast}$. The conditions on the weights of the relevant rows resp. columns are the same in both constructions. There remains the question of the forbidden rectangles. These rectangles are generated by $(l,k)$ and an opposite corner $(i,j)$, where $(i,j)$ lies in a certain area. But this area is the same in both constructions, namely
\[
\left\{(i,j)\mid 1\leq i\leq l\mbox{ and }1\leq j \leq k\right\}.
\]
Therefore $c_{lk}=a_{lk}$ and we obtain $a_{ij}=a_{ji}$ for $1\leq i,j$, i.e.
\begin{thm}\label{thm:3.2}The matrix $A$ is symmetric.
\end{thm}
\begin{lem}\label{lem:3.4a}
  {Let $i\geq 1$. There exists $j$ such that $1\leq j \leq i$ and $a_{ij}=1$}.
\end{lem}
\begin{proof} Suppose $a_{i 1}=\ldots=a_{i,i-1}=0$. Then by symmetry $a_{1i}=\ldots=a_{i-1,i}=0$. When constructing the row $A_{i\ast}$, we must put $a_{ii}=1$.
\end{proof}

\noindent\textbf{Remark.} (See \cite{KHE}). Let $k,r \in \mathbb{N}$. There exists exactly one matrix $A=(a_{ij})_{i,j\in\mathbb{N}}$ over $\{0,1\}$ such that $a_{ij} = 1$ if and only if none of the following conditions holds
\begin{itemize}
  \item There exist $\overline{i} \leq i$ and $\overline{j} \leq j$ such that $a_{i,\overline{j}} = 
a_{\overline{i},j} = a_{\overline{i},\overline{j}} = 1$
  \item $\sum_{\overline{j}<j} a_{i,\overline{j}} \geq k$
  \item $\sum_{\overline{i}<i} a_{\overline{i},j} \geq r$
\end{itemize}
This matrix is called the \emph{naive matrix} of Type $(k,r)$.
\subsection{The defining matrices}\label{definingmatrices}
In this section we prove that the matrix $A$ is periodic according to the following definition:
\begin{defn}\label{def:3.2}
If there exist integers $p \geq 1$ and $pp \geq 0$ such that
\[
a_{i+p,j+p} = a_{ij} ~for ~all ~i,j > pp,
\]
then we call the $\mathbb{N}\times\mathbb{N}-matrix ~A$ \emph{periodic}, $p$ a \emph{period} and $pp$ a \emph{preperiod} of $A$.
\end{defn}

Let $k\geq 2$ and let $\overline{A}$ be the $\mathbb{N}\times\mathbb{N}$-matrix which coincides with $A$ in it´s first $k-1$ rows, but is 0 otherwise. Let $l(k) = l \geq 1$ be the smallest number, such that the $l$-th  column $\overline{A}_{\ast l}$ contains fewer than $n+1$ ones. By Lemma \ref{lem:3.3}, $\overline{A}_{\ast l} = 0$ if and only if $\overline{A}$ vanishes on, and to the right of the column $\overline{A}_{\ast l}$.

\textbf{Case 1}. Assume that $\overline{A}_{\ast l}=0$. Clearly, $a_{k1}= \ldots=a_{k,l-1}=0$ as $\overline{A}_{\ast i}$ is complete for $1\leq i < l$. By symmetry (Theorem \ref{thm:3.2}), also $a_{1k}= \ldots=a_{l-1, k}=0$. Hence, $k\geq l$ by the minimality of $l$. On the other hand, by Lemma \ref{lem:3.4a}, there exists $r$ such that $1\leq r \leq l$ and $a_{rl}=1$. So $\overline{A}_{\ast l}=0$ implies $k\leq l$, and $k=l$. When we continue the construction of the matrix $A$ and construct $a_{kk}$, we have exactly the same situation as we had, when we were constructing $a_{11}$.
Therefore we have $a_{i+p,j+p}=a_{ij}$ for $i,j \geq 0$, where $p=k-1$, and $A$ is periodic. Hence for $pp=0$ and $p=k-1$ we have Theorem \ref{thm:3.3} below, except for the last inequality.

In the (more general) case, when $\overline{A}_{\ast l} \neq 0$, we determine for each $k\geq 2$ a finite $\{0,1\}$-matrix $M^k$, which determines $A_{k\ast}$ and all further rows $A_{i\ast}$ with $i\geq k$. We have an upper bound for the size of $M^k$. Therefore the matrices $M^k$ will repeat eventually and the matrix $A$ will be periodic after a certain preperiod.

\textbf{Case 2}. Assume that the column $\overline{A}_{\ast l}$ is not the zero column. Define
\[
b=c-l+1,\]
where $c$ is the largest number such that $a_{rc}=1$ for some $r$ such that $1\leq r < k$,
\[
d=k-f,\]
where $f$ is the smallest number such that $a_{fl}=1$, and
\[
M^k_{ij}=a_{i+f-1,j+l-1}\]
 for $1\leq i \leq d$ and $1 \leq j \leq b$.

By Lemma \ref{lem:3.3}, we know $a_{ij}=0$ for $i<k$ and $j>c=l+b-1$, and for $i<f=k-d$ and $j\geq l$. Therefore, the $d\times b$-matrix $M^k$ together with the parameters $k$ and $l$ completely determines the construction of the row $A_{k\ast}$ and all succeeding rows $A_{i\ast}$, $i\geq k$. We call $M^k$ the $k$th \emph{defining matrix}.

Now $b\leq 2n^3-n(n-3)$ by Theorem \ref{thm:3.1}, and $d\leq 2n^3-n(n-3)-1$ because, in addition, the matrix $A$ is symmetric by Theorem \ref{thm:3.2}.

(The ``height'' of a column in $A$ is limited, as is the ``length'' of a row. Also, by the construction of the row$A_{r,\ast}$, this row must contain a one on or to the left of the column $A_{\ast l}$, as $A_{\ast l}$ is not complete.)

So the size of the defining matrix $M^k$ is limited. Denote $\sigma = 2n^3-n(n-3)$ and let us consider all cases from $k=2$ up to $k=2^{\sigma^2}+1$. If for some $k\leq2^{\sigma^2}+1$ we have Case 1, then we obtain Theorem with $pp = 0$, $p = k-1$ and $pp + p = k-1 \leq 2^{\sigma^2}$. Assume now that Case 1 never occurs. Then two of the resulting matrices $M^k$ must be equal. Let $\overline{p}$ be the smallest number $\geq 1$, such that there exists $k\geq 2$ such that
\[
M^{k+\overline{p}}=M^k,
\]
where $k + \overline{p}\leq 2^{\sigma^2} + 1$, and let $\overline{pp}\geq 1$ be the smallest number such that
\[
M^{\overline{pp}+\overline{p}}=M^{\overline{pp}}.
\]

The matrix $M^k$ together with the parameters $k$ and $l=l(k)$ determine $A_{k\ast}$ and the part of the matrix below the row $A_{k\ast}$. By the symmetry of $A$ (Theorem \ref{thm:3.2}, also by Lemma \ref{lem:3.4a} and Theorem \ref{thm:3.1}), the ones in $A$ must remain close to the main diagonal. Therefore we have:
\begin{lem}\label{lem:3.4}
  $l(k+\overline{p})=l(k)+\overline{p}$.
\end{lem}
Hence we have $a_{i+\overline{p},j+\overline{p}}=a_{ij}$ for $i\geq k$. This proves 
\begin{thm}\label{thm:3.3}
There exist integers $pp$ and $p$ such that $0\leq pp,~1\leq p$, 
\[
a_{i+p,j+p}=a_{ij}~for~i > pp,~and
\]
\[
pp+p \leq 2^{\sigma^2},
\]
where $\sigma = 2n^3-n(n-3)$.
\end{thm}
Let $p$ be the smallest number such that there exists $c\geq 0$, such that $a_{i+p, j+p}=a_{ij}$ for $i>c$ and let $pp$ be the smallest number such that $a_{i+p, j+p}=a_{ij}$ for $i>pp$. We call $p=p(n)$ the \emph{period} and $pp=pp(n)$ the \emph{preperiod} for $n$.\\

\noindent\textbf{Remark.} Clearly the bound for $p(n)$ can easily be improved, e.g. since in view of Lemma \ref{lem:3.3}, the upper right hand area of the defining matrix is always 0.

Because of the periodicity, the `breadth' of the matrix $A$ is limited. That is, the ones of the matrix $A$ remain `close' to the main diagonal. We define
\[
b (n) =\max\left\{\vert j-i \vert \mid i,j\geq 1, a_{ij}=1\mbox{ and }i>pp\right\}.
\]

Also, the lengths of the rows of $A$ are limited. We define ~$l_{max}$ ~to be the \emph{maximum length} of a row $A_{i*}$ with $i> pp(n)$.

Hence, we can always compute the complete Edgar matrix $A$ after finitely many steps.
\section{The configuration $\Edg{n}$}
Let $p=p(n)$ and let $pm, ~m\geq 1$, be a multiple of the period such that $$\left[\frac{pm}{2}\right] > b(n).$$ Denote $\overline{p}=pm$ and $r = \left[\frac{\overline{p}}{2}\right]$. Furthermore, let $v$ be a rational integer larger than or equal to $pp(n)+\overline{p}$. We now define a new matrix $B=(b_{ij})$. $B$ is a $\overline{p}\times\overline{p}$-matrix, and the coefficients of $B$ are defined by
\[
 b_{ij}=
\begin{cases}
 a_{v+i, v+j} & \text{ if }  i-r \leq j \leq i+r\\
 a_{v+i, v+j-\overline{p}} & \text{ if }  j>i+r\\
 a_{v+i, v+j+\overline{p}} & \text{ if }  j<i-r\\
\end{cases}
\]
for $1\leq i,j \leq \overline{p}$.
\begin{thm}\label{thm:4.1}
The weight of every row of $B$ is $n+1$.
\end{thm}
\begin{proof}
  Let $1\leq i \leq \overline{p}$. The $i$th row $B_{i\ast}$ of $B$ is constructed from the $(v+i)$th row of $A$, which has weight $n+1$ by the construction of $A$. Part of $B_{i\ast}$ is obtained by shifting a segment of the row $A_{i\ast}$ to the right, respectively to the left. The complementary segment of $B_{i\ast}$ just remains the same as the corresponding segment of $A_{i\ast}$. This together with the inequality $b<\left[\frac{\overline{p}}{2}\right]$ implies that the weights of the rows remain the same.
\end{proof}
\begin{thm}\label{thm:4.2}
  B is symmetric.
\end{thm}
\begin{proof}
  Let $1\leq i,j \leq \overline{p}$. If $i-r\leq j \leq i+r$, then $b_{ij}=a_{v+i, v+j}=a_{v+j, v+i}$, because $A$ is symmetric by Theorem \ref{thm:3.2}. Also, $i-r\leq j \leq i+r$ implies $j-r\leq i \leq j+r$, so that $a_{v+j, v+i}=b_{ji}$ by definition of $B$. Finally, if $j<i-r$ (and $i>r$), then
  \[
  b_{ij}=a_{v+i,v+j+\overline{p}}=a_{v+i-\overline{p}, v+j}=a_{v+j, v+i-\overline{p}}=b_{ji},
  \]
  where we have the equalities because of the definition of the matrix $B$ by the periodicity of $A$, the symmetry of $A$, and, again, the definition of $B$.
\end{proof}
From Theorem \ref{thm:4.1} and Theorem \ref{thm:4.2} we obtain
\begin{thm}\label{thm:4.3}
The weight of every column of $B$ is $n+1$.
\end{thm}

Thus our new (finite!) matrix $B$ again fulfills Axioms (II) and (III). We now prove that $B$ also fulfills Axiom (I).

We use the matrix $B$ as the incidence matrix of an incidence structure. Let $\Edg{n} = \mathcal{E} = (\mathcal{P},\mathcal{L})$, where $\mathcal{P}$ is the set of columns and $\mathcal{L}$ the set of rows of $B$. Define the incidence
\[
B_{\ast i}\ I\ B_{j\ast} \Leftrightarrow b_{ji}=1 \text{ for } 1\leq i,j \leq \overline{p}
\]
$\mathcal{E}=\Edg{n}=(\mathcal{P,L})$ is the \emph{Edgar structure} of $n$. Clearly, $\Edg{n}$ is a finite symmetric tactical configuration. Also, we have
\begin{thm}\label{thm:4.4}
$\abs{(a)\cap (b)}\leq 1$ for $a,b\in\mathcal{L}$ and $a\neq b$, if we choose $m$ such that $$\overline p =pm \geq 2\cdot l_{max}.$$
\end{thm}
\begin{proof}Suppose we have integers $i,j,k,l$ such that $1\le i,j,k,l \le \bar p = pm,~ b_{jk} = b_{jl}=b_{il} = b_{ik}=1,~i < j$ and $k< l$. We denote the corners of the generated "rectangle" by $C_{1} = (j,k)$, $C_{2} = (j,l)$, $C_{3} = (i,l)$ and $C_{4} = (i,k)$. We denote 
\begin{itemize}
  \item $A_1 = \{(i,j) \mid 1\leq i,j \leq \overline{p}$ and $i-r \leq j \leq i+r\}$,
  \item $B_1 = \{(i,j) \mid 1\leq i,j \leq \overline{p}$ and $j>i+r\}$,
  \item $B_2 = \{(i,j) \mid 1\leq i,j \leq \overline{p}$ and $j<i-r\}$, and
  \item $B = B_1 \cup B_1 $.
\end{itemize}
For $1 \leq s \leq 4$ we define $c_{s}$ by 

$$c_{s} = \begin{cases}
1 & \text{if $C_s\in A_1$ and }\\
0 & \text{if $C_s\in B=B_1\cup B_2$.}\\
\end{cases}$$

So we obtain a vector $(c_{1}, c_{2}, c_{3}, c_{4}).$ We discuss the 16 possibilities for this vector. 

\begin{description}
\item[$(1111)$] Here all 4 corners lie in the original matrix $A.$ This is not possible by Axiom I.
\item[$(1110)$] We have 

\hspace{-.5cm}$(*)$ If a corner $C$ lies in $B_{1},$ then any corner above $C$ or to the right of $C$ lies in $B_{1}$.

    As $C_{4}$ lies in $B$ but $C_{3}$ not, we find that $C_4\notin B_{1}.$ On the other hand $C_4 \in B,$ but $C_1 \in A.$ Hence by the dual of $(*)$, $C_{4}\notin B_{2},$ a contradiction.
    Note that the dual of $(*)$ reads 

\hspace{-.5cm}$(*d)$ If a corner $C$ lies in $B_{2},$ then any corner below $C$ or to the left of $C$ lies in $B_2.$ 
\item[$(1101)$] By $(*d),\, C_{3}=(i,l) \in B_{1}$. Therefore $l>i+r$ and $b_{il} = a_{v+i,v+l-\overline{p}}$. We have $1 = b_{jk} = a_{v+j,v+k}$ and $1 = b_{jl} = a_{v+j,v+l}$
so that $l-k\leq l_{max}-1$. Also $1 = b_{il} = a_{v+i,v+l-\overline{p}}$ and $1 = b_{ik} = a_{v+i,v+k}$ so that 
$k-l+\overline{p}\leq l_{max}-1$. Hence $\overline{p}\leq 2\cdot (l_{max}-2)$.
\item[$(1100)$] By $(*d),$ $C_{3}$ and $C_{4}$ cannot belong to $B_{2}.$ So $C_{3}, C_{4} \in B_{1}.$ This case is symmetric to the case $(0110)$ below. 
\item[$(1011)$] Here $C_2 \in B.$ If $C_2 \in B_1,$ then $C_3 \in B_1$ by $(*)$. So $C_2 \in B_2,$ and $C_1 \in B_2$ by $(*d),$ a contradiction. 
\item[$(1010)$] leads to a contradiction as $(1011).$ 
\item[$(1001)$] Here $C_2,C_3 \in B_1$ by $(*d)$. We have $b_{ik} = a_{v+i,v+k} = b_{jk} = a_{v+j,v+k} =1$ and $b_{il} = a_{v+i,v+l- \overline p} = b_{jl} = a_{v+j,v+l- \overline p} = 1.$ This contradicts Axiom I. (Note that ${l \leq \overline p}$ so that $v+l- \overline p <v+k$.) 
\item[$(1000)$] $C_2 \in B$ and $C_1 \in A_1$ implies $C_2 \in B_1$ by $(*d).$ We have

\hspace{-0.5cm}$(**)$ If a corner $C$ belongs to $B_i , i \in \{1,2\},$ then any corner which lies in $B = B_1 \cup B_2$ and directly below, above, to the right or to the left of $C$ again lies in $B_i$. (Note that no row of $B$ contains a cell of $B_1$ and a cell of $B_2$ at the same time: Suppose we have $(i,j) \in B_1$ and $(i,j') \in B_2$. Then $j>i+r$ and $i<j-r \leq \overline{p}-r\leq r+1.$ On the other hand, $j'<i-r$ and $i>j'+r\geq 1+r$, a contradiction.) \\
By this Lemma $C_3 , C_4 \in B_1.$ Therefore $l>j+r$ and $k,l>i+r$. Analogous to the case (1101) we have $1 = b_{ik} = a_{v+i,v+k-\overline{p}}$ and $1 = b_{il} = a_{v+i,v+l-\overline{p}}$ so that $l-k\leq l_{max}-1$. Also $1 = b_{jl} = a_{v+j,v+l-\overline{p}}$ and $1 = b_{jk} = a_{v+j,v+k}$ so that 
$k-l+\overline{p}\leq l_{max}-1$. As in Case (1101) we obtain $\overline{p}\leq 2\cdot (l_{max}-2)$, contradicting our assumption.

\item[$(0111)$] $C_1 \in B_2$ by $(*) . $ This is symmetric to the case $(1101) . $ 
\item[$(0110)$] $C_1 \in B$ and $C_2 \in A_1$ implies $C_1 \in B_2$ by $(*) . $ In the same way we see that $C_4 \in B_2 . $ As in case $(1001)$ we can prove that this is impossible. 
\item[$(0101)$] $C_1 \in B$ and $C_2 \in A_1$ implies $C_1 \in B_2$. $C_3 \in B$ and $C_4 \in A_1$ implies $C_3 \in B_1 . $ \\
    We construct from our `rectangle' in the matrix $B$ a `rectangle' in the matrix $A$. We have \\
    $a_{v+i,v+l- \overline p}~=~b_{il}~=~1$,\\
    $a_{v+i,v+k}~=~b_{ik}~=~1$,\\
    $a_{v+j- \overline p,v+k}~=~a_{v+j,v+k+ \overline p}~=~b_{jk}~=~1$, and \\
    $a_{v+j- \overline p,v+l- \overline p}~=~a_{v+j,v+l}~=~b_{jl}~=~1$,\\
    because $C_3~=~(i,l) \in B_1 ,~C_4~=~(i,k) \in A_1 ,~C_1~=~(j,k) \in B_2 ,~C_2~=~(j,l) \in A_1$, and because $A$ is periodic with period $p.$ But this contradicts Axiom (I). (Note that $v+j- \overline p \neq v+i$ and $v+l- \overline p \neq v+k$ because $j-i,k-l< \overline p.$) Hence this case is not possible. 
\item[$(0100)$] $C_1 \in B$ and $C_2 \in A$ implies $C_1 \in B_2$ by $(*).$ By $(**)$ we obtain $C_4 \in B_2$ and $C_3 \in B_2.$ Thus $C_2 \in B_2$ by $(*d),$ a contradiction. 
\item[$(0011)$]$C_4 \in A$ and $C_1 \in B$ implies $C_1 \in B_2$ by $(*).$ Also $C_2 \in B_2$. This case is symmetric to $(1001)$ and hence impossible. 
\item[$(0010)$]$C_2 \in B$ and $C_3 \in A_1$ implies $C_2 \in B_2$ by $(*).$ Hence $C_1 , C_4 \in B_2$ by $(**).$ Symmetric to $(1000).$ 
\item[$(0001)$]$C_4 \in A_1$ and $C_1 \in B$ implies $C_1 \in B_2$ by $(*).$ Hence $C_2 , C_3 \in B_2$ by $(**).$ But $C_3 \in B_2$ and $C_4 \in A_1$ is impossible by $(*d).$ 
\item[$(0000)$] By $(**),~ C_1 , C_2 , C_3 , C_4 \in B_1$ or $C_1 , C_2 , C_3 , C_4 \in B_2.$ Impossible by Axiom (I). 
\end{description} \vspace{-0.75cm} \end{proof}
Hence, $\Edg{n}$ is a symmetric configuration with parameters $\overline{p}_{n+1}$ (i.e. with a point set of cardinality $\overline{p}$ and $n+1$ points on every line). Note that we obtain a symmetric configuration for every $m$ which is sufficiently large. Thus we actually have a series of configurations.

As a consequence of Theorem \ref{thm:4.4} we obtain
\begin{cor}\label{cor:4.5}
For every integer $k\geq 1$ there exists a finite symmetric configuration with k points on each line.
\end{cor}

\section{Examples}\label{examples}
We know the structure of $\Edg{n}$ for infinitely many $n$. In spite of this, 
actually computing the symmetric configuration $\Edg{n}$ is, in general, not so simple. For relatively small $n$ the computations already become extremely unwieldy.
Take $n=3$. In this case, the preperiod is 48, and the period is 16. $\Edg{3}$ is a symmetric configuration with parameters $16_4$ as defined in Gropp \cite{Gro}. (See Fig. \ref{fig:5.1}, for the Martinetti graph (see Gropp\cite{Gro}), Fig. \ref{fig:5.2}). Its automorphism group has order 2.

\begin{figure}[t]
\begin{center}
\includegraphics[trim=0cm 20cm 12cm 0, width=8cm]{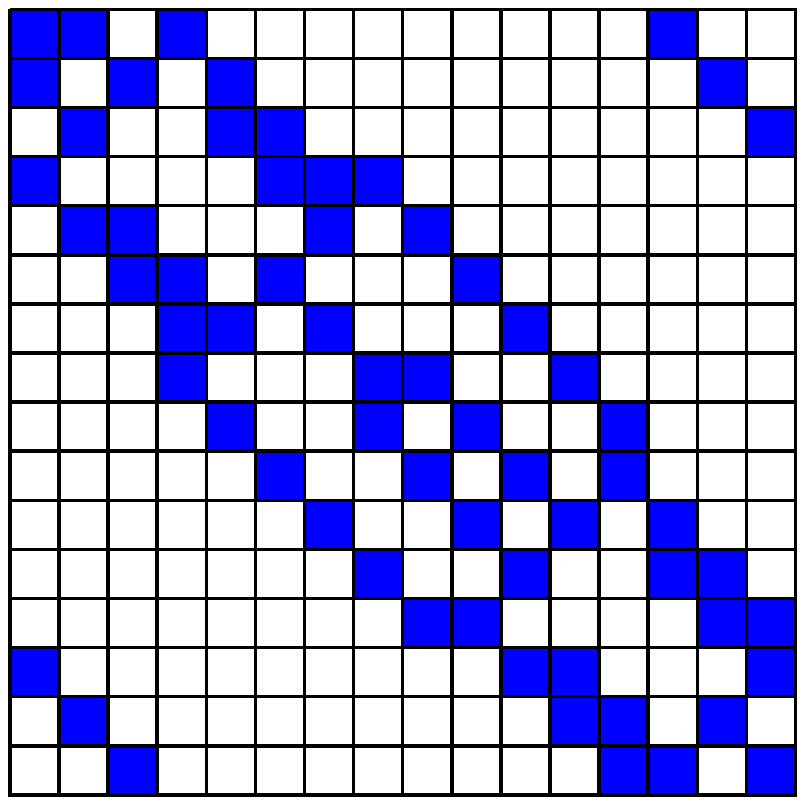}
\end{center}
\caption{Incidence matrix of $\Edg{3}$}\label{fig:5.1}
\end{figure}

\begin{figure}
\includegraphics[trim=4cm 15cm 0 6cm]{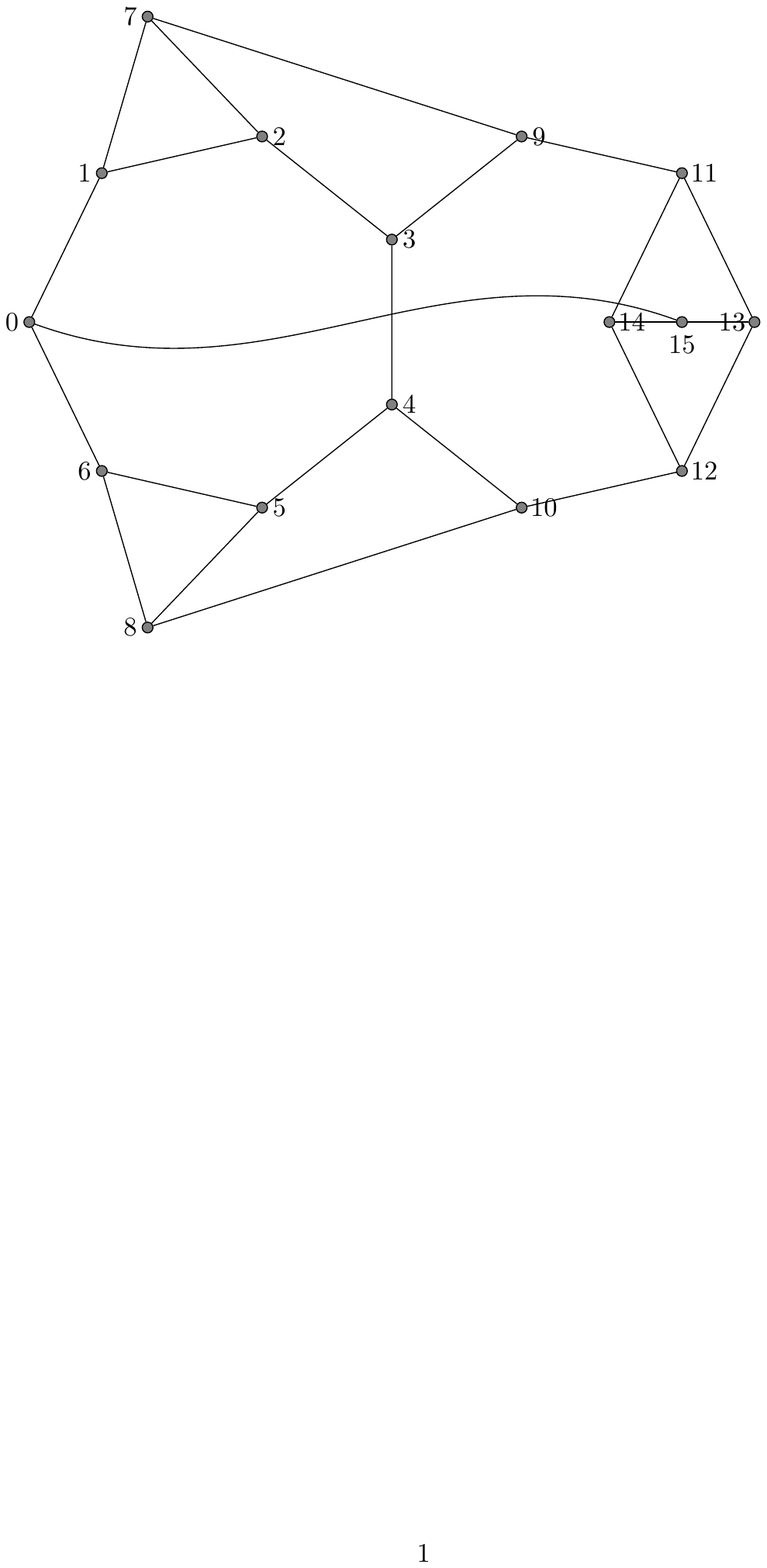}
\caption{Martinetti graph of $\Edg{3}$}\label{fig:5.2}
\end{figure}

For $n=1,~2,~4$, and 16, we find that the preperiod is 0 and the period is $p = n^2+n+1$. For these orders (where the Case 1 in Section 3.4 actually occurs) we can use a more compact construction replacing $\Edg{n}$. We just take $\overline{B}=(\overline{b}_{ij})$ to be the $p\times p$-matrix with coefficients $\overline{b}_{ij} = a_{ij}$ for $1\leq i,j \leq p$. Taking $\overline{B}$ as incidence matrix we construct an incidence geometry $\overline{\mathcal{E}}=\overline{\mathcal{E}}(n)$ as above. Actually, for $n \geq 2$ the incidence structure $\overline{\mathcal{E}}(n)$ in these cases is just a projective plane of order $n$ which turns out to be desarguesian. For $n=1$, we obtain just a triangle, hence a degenerate projective plane. Here
\[
\overline B =\begin{bmatrix}
    1 & 1 & 0 \\
    1 & 0 & 1 \\
    0 & 1 & 1 \\
  \end{bmatrix}.
\]
The original Edgar structure $\Edg{n}$ for $m=2$ and a suitable $v$ in these cases is just the `union', in some sense, of two copies of $\overline{\mathcal{E}}(n)$.

Note hat thereby we have a completely geometric simple construction of, e.g., the Galois field GF(16).

In a forthcoming paper \cite{KHE} we shall prove that, for every Fermat 2-power, that is, a number of the format $n=2^{2^a}$ for $a\geq 0$, $\overline{\mathcal{E}}(n) \cong$ PG(2, $n$).

Apparently the system favors Fermat 2 powers. Until now we could not discover, why this is the case.

As a further example, calculating $n=5$ takes quite a deal of patience. Hans-Joerg Schaeffer calculated that the preperiod is at least 5,652,533. However, $\Edg{5}$ holds a surprise, as we find that the period is just 31, and $\Edg{5}$ is isomorphic to the projective plane of order 5.

\label{sixorten}Now of course it would be very interesting to calculate $\Edg{6}$ and $\Edg{10}$ for example, because it is known that projective planes of these orders do not exist (by Euler \cite{Eul}, Lam \cite{Lam}, and MacWilliams, Sloane and Thompson \cite{Tho}). Unfortunately, we did, until today, not succeed in determining $\Edg{10}$. For $n=6$, we could find an upper bound for the preperiod of 15 trillion lines, while the period is $1,411,455,772,046=1,335,167*9,973*53*2$.

Since we cannot store even a sparse matrix with 15 trillion lines and 7 points in each line, we store only a small amount of lines (about 1 million), which enables us to compute the next line, while forgetting the oldest lines. This allows us to search for cycles up to a length of about one million, which is sufficient up to order 5. For order 6, we use the cycle finding algorithm of Floyd \cite{Flo}. The key idea of this algorithm is to compute for every $k$, the rows in the intervals $[k,k
+n^2+n+1]$ and $[2k,2k+2(n^2+n+1)]$ and check if the rows in the first interval appear in the second interval. 

It would be extremely interesting if $\Edg{10}$ could somehow be calculated.

Many more interesting cases, some of which emerge after a rather short computation, can be found if we extend our investigations to non-symmetric configurations (see \cite{KHE}).\\

\noindent\textbf{Remark.} Every $\{0,1\}$-matrix determines it´s \emph{galf-matrix} $\mathcal{G}(X)$. This matrix can be computed using the program set ProjFinder \cite{P}, and so we obtain an easy way to determine the matrix $A(n)$.\\

\end{document}